\documentclass[11pt]{article}
\pdfoutput=1 

\usepackage[english]{babel}

\usepackage{geometry}
\usepackage{amsmath}
\numberwithin{equation}{section} 

\usepackage[margin = 1ex]{subfig}

\usepackage{enumitem}

\usepackage{graphicx}\graphicspath{{./}{Figures/}}

\usepackage{amsthm}

\newtheoremstyle{italic}
{5pt}
{5pt}
{\itshape}
{}
{}
{}
{.5em}
{\bfseries{\thmname{#1}~\thmnumber{#2}.}\thmnote{~\textnormal{(#3)}}}

\newtheoremstyle{italic_borrowed}
{5pt}
{5pt}
{\itshape}
{}
{}
{}
{.5em}
{\bfseries{\thmname{#1}~\thmnumber{#2}.}\thmnote{~\textnormal{[#3]}}}

\newtheoremstyle{upright}
{5pt}
{5pt}
{\upshape}
{}
{\bfseries}
{}
{.5em}
{\bfseries{\thmname{#1}~\thmnumber{#2}.}\thmnote{~\textnormal{(\textit{#3}\textrm{)}}}}

\theoremstyle{italic}
\newtheorem{theorem}{Theorem}[section]
\newtheorem{lemma}[theorem]{Lemma}

\newtheorem{corollary}[theorem]{Corollary}

\theoremstyle{italic_borrowed}

\theoremstyle{upright}

\newtheorem{remark}[theorem]{Remark}

\usepackage{dsfont}
\usepackage{amsfonts}

\usepackage{xparse}

\usepackage{mathtools}

\newcommand\blfootnote[1]{%
  \begingroup%
  \renewcommand\thefootnote{}\footnote{#1}%
  \addtocounter{footnote}{-1}%
  \endgroup%
}%

\newcommand{\dconv}{\overset{d}{\to}}
\newcommand{\euler}{\mathrm{e}}
\newcommand{\BigO}{\mathrm{O}}

\newcommand{\argsup}{\arg\sup} 

\newcommand{\dinf}{\textup{d}} 

\newcommand{\Nat}{\mathbb{N}} 
\newcommand{\Real}{\mathbb{R}} 

\NewDocumentCommand \E { m g }{%
    \IfNoValueTF{#2}
        {\mathbb{E}[#1]}
        {\mathbb{E}_{#1}[#2]} 
    }%
\NewDocumentCommand \Esc { m g }{%
    \IfNoValueTF{#2}
        {\mathbb{E}\left[#1\right]}
        {\mathbb{E}_{#1}\!\!\left[#2\right]} 
    }%
\NewDocumentCommand \Efxd { m g }{%
    \IfNoValueTF{#2}
        {\mathbb{E}\Bigl[#1\Bigr]}
        {\mathbb{E}_{#1}\!\Bigl[#2\Bigr]} 
    }%
\NewDocumentCommand \Prob { m g }{%
    \IfNoValueTF{#2}
        {\mathbb{P}(#1)}
        {\mathbb{P}_{#1}(#2)} 
    }%
\NewDocumentCommand \Probfxd { m g }{%
    \IfNoValueTF{#2}
        {\mathbb{P}\Bigl(#1\Bigr)}
        {\mathbb{P}_{#1}\!\Bigl(#2\Bigr)} 
    }%

\newcommand{\bld}[1]{\mathbf{#1}} 
\newcommand{\oneb}{\bld{1}} 
\NewDocumentCommand \eb { m g }{%
    \IfNoValueTF{#2}
        {\bld{e}_{#1}}
        {\bld{e}^{(#1)}_{#2}} 
    }%
\NewDocumentCommand \zerob { g }{%
    \IfNoValueTF{#1}
        {\bld{0}}
        {\bld{0}^{(#1)}} 
    }%

\newcommand{\la}{\lambda} 

\newcommand{\be}{\beta}
\newcommand{\ga}{\gamma}
\newcommand{\La}{\Lambda}

\newcommand{\seq}{^{(s)}} 

\newcommand{\lseq}{\la\seq}

\newcommand{\pd}{\Prob{W > 0}}
\newcommand{\pdseq}{\Prob{W\seq > 0}}
\newcommand{\X}[1]{X(#1)}
\newcommand{\Xseq}[1]{X\seq(#1)}

\newcommand{\Dseq}[1]{D\seq(#1)}
\newcommand{\Y}[1]{Y(#1)}
\newcommand{\Yseq}[1]{Y\seq(#1)}

\newcommand{\F}{_{\scriptscriptstyle \textup{H}}} 
\newcommand{\mF}{\mu\F} 
\newcommand{\mFseq}{\mF\seq} 
\newcommand{\rF}{\rho\F} 
\newcommand{\rFseq}{\rF\seq} 
\newcommand{\pdF}{\Prob{W\F > 0}} 
\newcommand{\XFseq}[1]{X\F\seq(#1)} 
\newcommand{\DF}[1]{D\F(#1)} 
\newcommand{\DFseq}[1]{D\F\seq(#1)} 
\newcommand{\pdfF}[1]{f_{D\F}(#1)} 
\newcommand{\pdfCF}{C\F} 
\newcommand{\SSFseq}{E\F\seq} 
\newcommand{\meanDF}{m\F(x)} 
\newcommand{\meanDFseq}{m\F\seq(x)} 
\newcommand{\varDF}{\sigma\F^2(x)} 
\newcommand{\varDFseq}{\bigl( \sigma\F^2 \bigr)\seq(x)} 

\newcommand{\SL}{_{\scriptscriptstyle \textup{L}}} 
\newcommand{\mS}{\mu\SL} 
\newcommand{\rS}{\rho\SL} 
\newcommand{\rSseq}{\rS\seq} 
\newcommand{\pdS}{\Prob{W\SL > 0}} 
\newcommand{\XSseq}[1]{X\SL\seq(#1)} 
\newcommand{\DS}[1]{D\SL(#1)} 
\newcommand{\DSseq}[1]{D\SL\seq(#1)} 
\newcommand{\pdfS}[1]{f_{D\SL}(#1)} 
\newcommand{\pdfCS}{C\SL} 
\newcommand{\meanDS}{m\SL(x)} 
\newcommand{\meanDSseq}{m\SL\seq(x)} 
\newcommand{\varDS}{\sigma\SL^2(x)} 
\newcommand{\varDSseq}{\bigl( \sigma\SL^2 \bigr)\seq(x)} 
\newcommand{\BS}{B\SL} 
\newcommand{\WS}{W\SL} 

\newcommand{\QEDl}{\be} 
\newcommand{\QEDmF}{\ga} 

\newcommand{\I}{I} 

\title{The snowball effect of customer slowdown in critical many-server systems}%
\author{Jori Selen\footnotemark[1] \footnotemark[2], Ivo J.B.F. Adan\footnotemark[1] \footnotemark[2], Vidyadhar G. Kulkarni\footnotemark[3], Johan S.H. van Leeuwaarden\footnotemark[2]}%

\begin{document}%

\maketitle%
\renewcommand{\thefootnote}{\fnsymbol{footnote}}%
\footnotetext[1]{Department of Mechanical Engineering, Eindhoven University of Technology, The Netherlands}%
\footnotetext[2]{Department of Mathematics and Computer Science, Eindhoven University of Technology, The Netherlands}%
\footnotetext[3]{Department of Statistics and Operations Research, University of North Carolina, Chapel Hill}
\renewcommand{\thefootnote}{\arabic{footnote}}%
\blfootnote{E-mail address: {\tt j.selen@tue.nl}}%
\renewcommand{\thefootnote}{\arabic{footnote}} \setcounter{footnote}{0}%

\begin{abstract}%
Customer slowdown describes the phenomenon that a customer's service requirement increases with experienced delay. In healthcare settings, there is substantial empirical evidence for slowdown, particularly when a patient's delay exceeds a certain threshold. For such threshold slowdown situations, we design and analyze a many-server system that leads to a two-dimensional Markov process. Analysis of this system leads to insights into the potentially detrimental effects of slowdown, especially in heavy-traffic conditions. We quantify the consequences of underprovisioning due to neglecting slowdown, demonstrate the presence of a subtle bistable system behavior, and discuss in detail the snowball effect: A delayed customer has an increased service requirement, causing longer delays for other customers, who in turn due to slowdown might require longer service times.
\end{abstract}%


\section{Introduction}%
\label{sec:introduction}%

The phenomenon of customer slowdown describes the fact that a customer's service requirement increases with the customer's experienced delay. While the operations management literature is largely built on the assumption that service times are independent of delay, a growing number of empirical studies, predominantly in healthcare settings, provide evidence for situations where slowdown occurs. This empirical evidence calls for the development of stochastic models that take into account slowdown, in order to not only assess its impact on the performance of service operations, but also to gain understanding of the fundamental changes that slowdown brings to system behavior.

A large body within the healthcare operations literature investigates the impact of workload on service times of patients. A canonical example in this domain is the admission of patients to the intensive care unit (ICU). There is substantial empirical evidence for slowdown in such settings: delays in receiving appropriate care can result in adverse effects such as an increased length of stay in the ICU \cite{Empirical_Chalfin2007,Slowdown_Chan2013,Empirical_Chan2008,Empirical_Renaud2009,Empirical_Richardson2002,Empirical_Siegmeth2005}. Since ICUs are typically heavily used and subject to unforeseen circumstances, delays in admitting patients are the rule rather than the exception, which makes the slowdown effect potentially threatening. A delayed patient that requires a longer service time will increase the overall workload of the system, therefore causing longer delays for other patients, who in turn due to slowdown might require longer service. This triggers a \textit{snowball} effect, with an impact that is hard to assess without having a detailed understanding of the global system behavior that takes into account the subtle dependencies among customers due to slowdown. Particularly when a system like an ICU is designed to operate under heavy-traffic conditions, the neglect of slowdown might lead to underprovisioning and severe performance degradation.

Critical care systems such as an ICU are typically modeled as multi-server systems that operate in heavy-traffic regimes \cite{HealthcareQueuingTheory_Armony2011,HealthcareQueuingTheory_Green2006}. The patients are the customers, the beds are the servers, and the performance analysis of the multi-server systems gives insight into the patient flow. We shall consider a Markovian multi-server system with the additional feature of slowdown. A detailed analysis of this system gives insight into the key features of slowdown, in particular when compared against multi-server systems without slowdown.


\subsection{A threshold slowdown system}%
\label{subsec:modeling_slowdown}%

Slowdown can be modeled as a non-increasing function $\mu(\cdot)$ that describes the rate of service as a function of the queue length seen upon arrival. That is, a customer meeting $n$ customers upon arrival will receive service with rate $\mu(n)$, regardless of arrivals and departures after the customer has joined the system. Note that, since the number of customers seen on arrival can be translated into an expected delay, the service rate can also be interpreted as a non-increasing function of the expected delay. Assuming a service rate that is a function of the state of the system, leads to a so-called state-dependent queueing system.

The majority of the empirical studies on slowdown has focused on a threshold slowdown: if a patient's delay surpasses a certain threshold, he will receive a longer service time and otherwise he receives a service of regular length. In terms of the slowdown function, this means that $\mu(n) = \mF$ if $n \le N$ and $\mu(n) = \mS$ if $n > N$ with $\mF > \mS$. Note that the expected delay is translated to a number of customers $n$ met on arrival and compared against the threshold $N$. The definition of the threshold varies across different medical conditions and situations. In \cite{Empirical_Chalfin2007} it is argued that a critically ill patient awaiting transfer from the emergency department to the ICU is labeled as \textit{delayed} if the patient has waited longer than 6 hours. Delayed patients on average have an ICU length of stay that is 1 full day longer than the non-delayed average length of stay. Similar conclusions are drawn in \cite{Empirical_Richardson2002} for the same situation in different hospitals. However, \cite{Empirical_Richardson2002} uses a threshold of 8 hours. Both studies \cite{Empirical_Chalfin2007,Empirical_Richardson2002} establish a strong correlation between the delay a patient experiences in receiving an assigned bed and the ICU length of stay. Depending on the medical condition, the delay threshold can be in the order of minutes, such as for cardiac arrest patients \cite{Empirical_Chan2008}, hours, as seen in \cite{Empirical_Chalfin2007,Empirical_Richardson2002}, or even days, such as the 2 day delay in receiving surgery \cite{Empirical_Siegmeth2005}; or a 3 day threshold of delay for pneumonia patients \cite{Empirical_Renaud2009}. An encompassing study is performed in \cite{Slowdown_Chan2013}, where it is empirically verified that the slowdown effect is prevalent across multiple hospitals and patient conditions.

We shall adopt the model in \cite{Slowdown_Chan2013}, which is a multi-server model with a threshold service rate function $\mu(\cdot)$. Customers arrive according to a Poisson process with rate $\la$, have an exponential service requirement, and are served by $s$ servers. Due to the threshold, we then distinguish between two types of customers: those who were taken into service immediately upon arrival (non-delayed) and those who have experienced delay (delayed). We set the threshold to $N = s$ so that non-delayed customers are served with a high service rate $\mF$ and delayed customers are served with a low service rate $\mS$ with $\mF > \mS$. Indeed, in that case, delays cause a longer service time. Define the two-dimensional Markov process $(\X{t},\Y{t})$, with $\X{t} \in \Nat_0$ the total number of customers in the system at time $t$, and $\Y{t} \in \{0,\ldots,s\}$ the number of non-delayed customers in service at time $t$. Denote $\rF = \la / (s\mF)$ and $\rS = \la / (s\mS)$. The system is stable when $\rS < 1$. When $\mF = \mS$ the model reduces to the standard $M/M/s$ system. The load of the slowdown system is described by $\rho = (1 - \Prob{W > 0})\rF + \Prob{W > 0}\rS$, where $W$ is the stationary waiting time, so that $\Prob{W > 0}$ is the delay probability. Figure~\ref{fig:trd_abstract} displays the state space and the transition rate diagram. In \cite{Slowdown_Chan2013}, approximations are derived for key performance indicators that give insight into the slowdown effect. Based on parameter values calibrated from real ICU dataflows, the approximations in \cite{Slowdown_Chan2013} indicate that the slowdown effect can be substantial, and should not be ignored in critical care systems that operate in heavy traffic.

\begin{figure}%
\centering%
\includegraphics{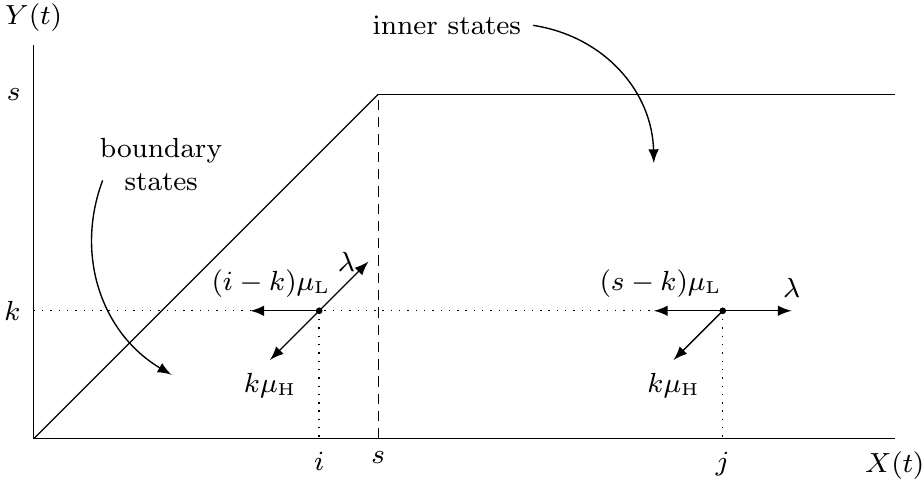}
\caption{Transition rate diagram and state space of the threshold slowdown system.}%
\label{fig:trd_abstract}%
\end{figure}%

This two-dimensional Markov process can be used to investigate the impact of slowdown, both qualitatively and quantitatively, in particular in comparison with the widely applied $M/M/s$ system (which neglects slowdown). While the focus in \cite{Slowdown_Chan2013} lies on approximations for small to moderate-sized systems (ICUs of 6 and 15 beds), we focus on exact and asymptotic results, both for finite $s$ and the regime $s \to \infty$ and $\rho \uparrow 1$.  With exact results we refer to determining the stationary distribution of the Markov process using a numerically stable algorithm. This algorithm allows us to compute the exact two-dimensional stationary distribution for not only small but also large systems. Asymptotic results give rise to accurate approximations for the dimensioning of large systems in heavy traffic. Particularly, we are interested in the effect of slowdown in the Quality-and-Efficiency driven (QED) regime \cite{QED_HalfinWhitt1981}. As it turns out, the way to establish non-degenerate limiting behavior for a multi-server system with slowdown in a QED-type regime is by letting $\rS$ approach 1, and $\mF$ approach $\mS$ as $s \to \infty$. We find that this scaling window is such that the probability of delay converges to a value that lies strictly in the interval $(0,1)$, which is a manifestation of non-degenerate limiting behavior. As pointed out in \cite{Slowdown_Chan2013}, deriving exact results becomes mathematically challenging because determining the stationary distribution of the Markov process involves high-dimensional matrix inversion. To relieve this computational burden of a large state space (particularly for large $s$), we exploit the fact that the Markov process has a block diagonal structure in the \textit{inner} states (states with more than $s$ customers in the system), which allows for an exact solution using matrix-analytic techniques. This technique typically relies on iterative algorithms that solve a non-linear matrix equation. For our model, we are able to find an exact solution for this matrix equation, which then immediately renders the problem of computing the stationary probabilities of the inner states computationally tractable, also for large $s$, see Section~\ref{subsec:inner_equations}. What remains is the computation of the stationary probabilities of the \textit{boundary} states (states with $s$ or less customers in the system). We introduce a novel approach that computes the exact stationary probabilities of the boundary states by exploiting the transition structure and by introducing first-passage probabilities. A detailed description of this approach can be found in Section~\ref{subsec:boundary_equations}.


\subsection{On the relation with operator slowdown}%
\label{subsec:relation_operator_slowdown}%

Slowdown can refer to \textit{customer} slowdown and \textit{operator} slowdown. Customer slowdown refers to an increase of a customer's service requirement, caused by the delay experienced by that customer. Operator slowdown refers to a service rate that decreases with the workload present in the system. Operator slowdown usually occurs in large service systems, such as call centers, due to fatigued operators \cite{Slowdown_Dong2014}. However, it is also common in medical applications under high workload, where care providers have to multitask and share (now crowded) central resources such as computer terminals \cite{Empirical_Batt2012}. The key difference is that customer slowdown starts with an individual delayed customer, and affects all customers behind this customer, while a decreased service rate in operator slowdown affects all customers that are in service. Customer slowdown therefore typically requires a more detailed state description, making it harder to analyze than operator slowdown. In this paper we indeed focus on customer slowdown, but we make comparisons with operator slowdown in several places.

Operator slowdown under Markovian assumptions leads to a one-dimensional Markov process which is more tractable than our two-dimensional process and is amenable to fluid analysis. In \cite{Slowdown_Dong2014} an $M/M/s$-type model with operator slowdown is investigated. Additional properties in \cite{Slowdown_Dong2014} are customer abandonments and state-dependent service rates. We make a comparison with \cite{Slowdown_Dong2014} by extending our base model to also include customer abandonments in Section~\ref{subsec:bistability}. Both slowdown models exhibit a bistable behavior in which the models alternate between two dominant regions. For the customer slowdown model, however, this behavior is more subtle than for the operator slowdown model (see Section~\ref{subsec:bistability}).

Both customer and operator slowdown fall into the broad category of queueing systems with state-dependent service rates, like for instance an $M/G/1$ system with state-dependent service rates \cite{StateDependent_Harris1967}. In \cite{StateDependent_Bekker2006}, the optimal admission policy is studied for an $M/G/1$ system with service rates that increase with the workload below a certain threshold and decrease with the workload above this threshold. State-dependent queueing systems also arise when arrival and/or service rates are dynamically controlled to minimize average cost per time unit, see e.g. \cite{StateDependent_Ata2006,StateDependent_George2001,StateDependent_Weber1987}. All these examples concern operator slowdown.


\subsection{Structure of the paper}%
\label{subsec:structure_paper}%

The paper is structured as follows. Based on a detailed analysis of the two-dimensional Markov process in Figure~\ref{fig:trd_abstract}, we identify three key features of threshold slowdown systems: severe performance degradation due to the snowball effect; a subtle bistable system behavior; and the existence of non-degenerate limiting behavior in a QED-type heavy-traffic regime. We discuss these three features in Section~\ref{sec:key_features}. The first two features were identified by using the stationary distribution of the two-dimensional Markov process. Section~\ref{sec:model} introduces the model in greater detail and Section~\ref{sec:obtaining_stationary_distribution} describes how we solve for its stationary distribution using matrix-analytic methods and some properties of regenerative processes. The QED-type heavy-traffic regime is outlined in Section~\ref{sec:scaling limits}. We conclude in Section~\ref{sec:conclusion} and present some supporting results in the appendix.


\section{Key features of threshold slowdown systems}%
\label{sec:key_features}%

Unless stated otherwise, we assume a stable system, i.e. $\rS < 1$.


\subsection{Performance degradation due to the snowball effect}%
\label{subsec:performance_degradation}%

We first present a detailed description of the snowball effect caused by slowdown and then assess the adverse effects for system performance.

For explanation purposes, we refer with \textit{busy} periods and \textit{idle} periods to the excursions of the process $X(\cdot)$ above and at level $s$, and below level $s$, respectively. Hence, during busy periods, newly arriving customers will experience delay and are thus subject to slowdown. The snowball effect sets in each time a new busy period starts. An example sample path of idle and busy periods is given in Figure~\ref{fig:sample_path_queue_length}, where we plot the total number of customers in the system $L(t)$ at time $t$. Compared with an $M/M/s$ system without slowdown (with a high service rate $\mF$), the busy period in the time interval $(9300,10000)$ is relatively long, due to the slowdown of delayed customers that reinforces, through other delayed customers, the persistence of the busy period. Such busy periods are essentially equivalent to busy periods in an $M/M/s$ system with a low service rate $\mS$. These excursions during which congestion levels are high occur relatively frequently due to the snowball effect that triggers them, and this leads to severe performance degradation, particularly in heavy traffic.

\begin{figure}%
\centering%
\includegraphics{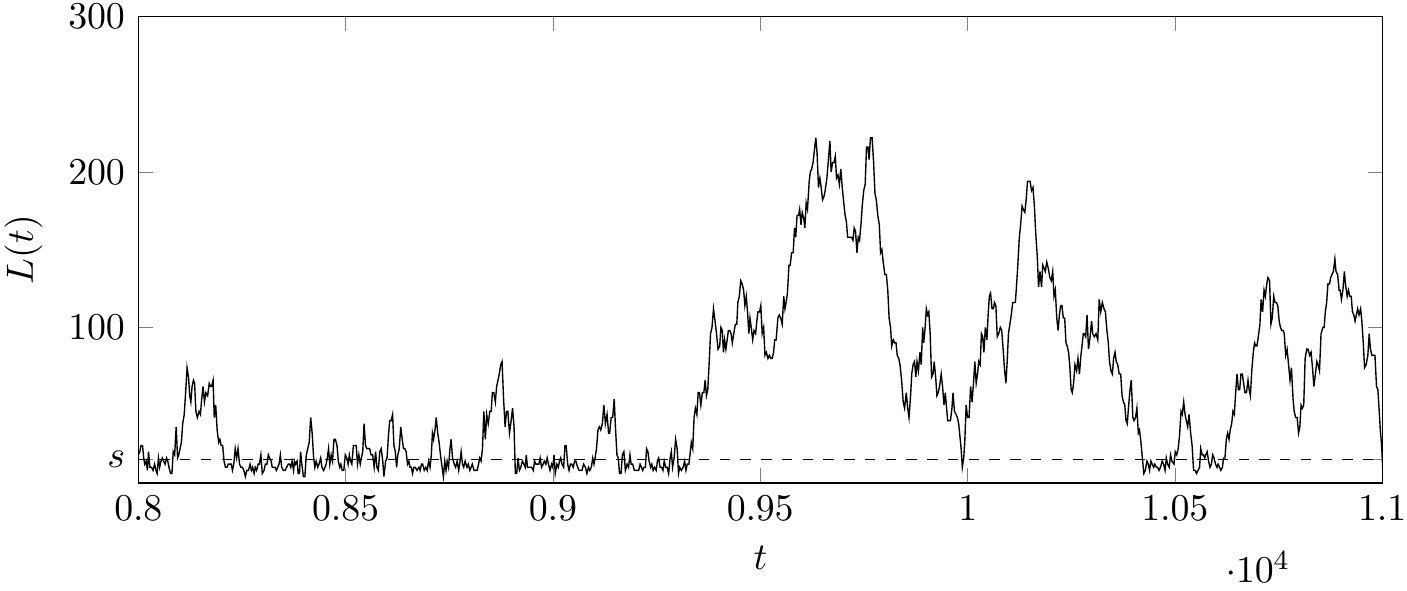}
\caption{Sample path of the total number of customers in the system $L(t)$ as a function of time $t$, for $s = 15$, $\la = s$ and loads $\rF = 0.7$ and $\rS = 0.98$.}%
\label{fig:sample_path_queue_length}%
\end{figure}%

This performance degradation is visible in Figure~\ref{fig:p_marginal}, which displays for the same parameter values as in Figure~\ref{fig:sample_path_queue_length} the stationary distribution of the total number of customers in the system $L$ of the threshold slowdown system. This stationary distribution is calculated using the numerical scheme that will be discussed in Section~\ref{sec:obtaining_stationary_distribution}. We also plot the stationary distribution of an $M/M/s$ system with uniform service rate $\mF$ (the \textit{fast} system) and with uniform service rate $\mS$ (the \textit{slow} system). We append the subscripts H, or L to random variables to indicate that they belong to the $M/M/s$ system with high service rate, or low service rate, respectively. We see that the distribution of the slowdown system peaks around the same point as the fast system, but that the tail behavior of the slowdown system is more comparable to the slow system (which can be attributed to the snowball effect and long busy periods). Such a fat tail obviously has severe consequences for performance, and in Figure~\ref{fig:p_marginal} we see for instance that the mean number of customers in the system increases considerably due to slowdown.

\begin{figure}%
\centering%
\includegraphics{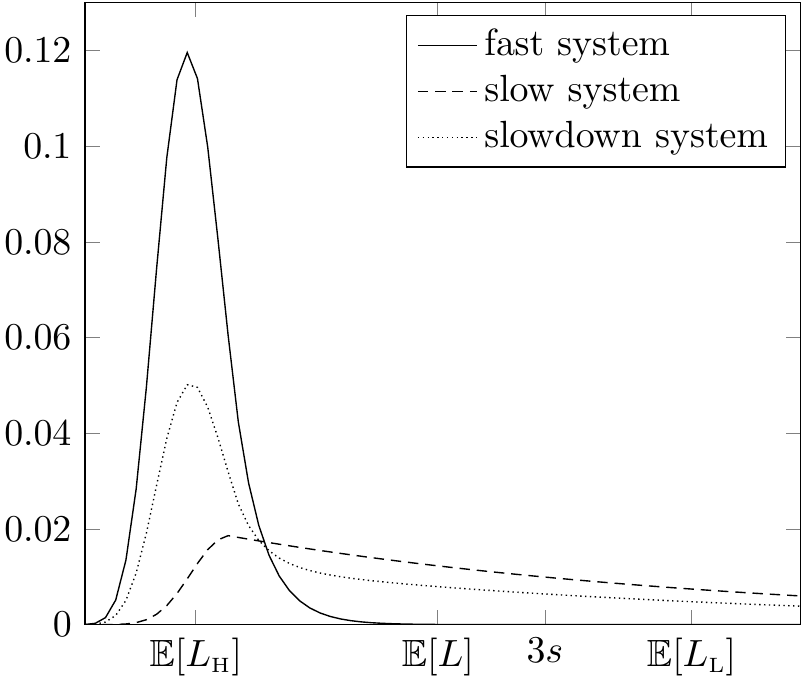}
\caption{Stationary distributions of the total number of customers in the system $L$ for systems with and without slowdown. Parameter values are $s = 15$, $\la = s$, $\rF = 0.7$ and $\rS = 0.98$. The expected number of customers in the systems are $\E{L\F} \approx 10.8$, $\E{L\SL} \approx 59.4$ and $\E{L} \approx 34.5$.}%
\label{fig:p_marginal}%
\end{figure}%



The neglect of slowdown might lead to underprovisioning. Table~\ref{tbl:design_servers} provides an example in which we search for the number of servers $s$ that are required to achieve a certain delay probability. Naturally, the threshold slowdown model requires equally many or more servers as required by the fast system with uniform service rate $\mF$. In particular, differences between the required number of servers in the fast model and the threshold slowdown model seem to increase with the delay probability, with the ratio $\mF/\mS$ and with the arrival rate $\la$.

\begin{table}%
\centering%
\begin{tabular}{ccc|ccc|ccc}%
$\mF$ & $\mS$ & $\la$ & $\Prob{W > 0}$ & $s\F^*$ & $s^*$ & $\Prob{W > 0}$ & $s\F^*$ & $s^*$ \\
\hline \hline%
1 & 0.9 & 10 & 0.1 & 16 & 16 & 0.5 & 12 & 13 \\
  &     & 12 &     & 18 & 18 &     & 14 & 15 \\
  &     & 15 &     & 22 & 22 &     & 18 & 19 \\
  &     & 20 &     & 27 & 28 &     & 23 & 24 \\
\hline%
1 & 0.7 & 10 & 0.1 & 16 & 17 & 0.5 & 12 & 15 \\
  &     & 12 &     & 18 & 19 &     & 14 & 18 \\
  &     & 15 &     & 22 & 23 &     & 18 & 22 \\
  &     & 20 &     & 27 & 30 &     & 23 & 29 \\
\hline%
\end{tabular}%
\caption{Minimal number of servers $s\F^*$ (fast system) and $s^*$ (threshold slowdown system) required to achieve a certain $\Prob{W > 0}$.}%
\label{tbl:design_servers}%
\end{table}%

Another indicator for substantial slowdown effect is the difference $\rho - \rF$, as we will show in the next subsection. This difference is the increase in load caused by the slowdown effect with respect to the load of the fast system.


\subsection{A subtle bistable behavior}%
\label{subsec:bistability}%

The threshold slowdown system behaves as the fast system below the threshold, and as the slow system above the threshold. However, for many relevant parameter settings, neither the fast nor the slow system provides a good approximation for the slowdown system. The reason is that the slowdown system in fact is a rather intricate mixture of both system as will be explained in this subsection.

We start by examining the two-dimensional stationary distribution, which typically consists of two dominant regions: region 1 with only non-delayed customers and no customers waiting, and region 2 with delayed (slowdown) customers in service only and many waiting customers. Region 1 thus complies with the fast system and region 2 with the slow system. An important parameter that determines whether region 1 or region 2 is dominant is $\rF$. A low to moderate $\rF$ makes region 1 dominant, which suggests using the fast system as a proxy. A high load $\rF$ makes region 2 more important, and in fact, when $\rF$ approach 1, the slow system will be a good approximation. See the two examples in Figure~\ref{fig:no_bistability}. Notice here that for a system with a high delay probability, i.e.~Figure~\ref{fig:no_bistability}(b), the increase in load $\rho - \rF$ due to the slowdown effect is small, since both loads $\rF$ and $\rS$ are large and comparable. In contrast, the increase in load in Figure~\ref{fig:no_bistability}(a) is much larger.

\begin{figure}%
\centering%
\subfloat[$\rF = 0.6$, $\rho - \rF = 0.123$ and $\Prob{W > 0} = 0.32$.]{%
\includegraphics{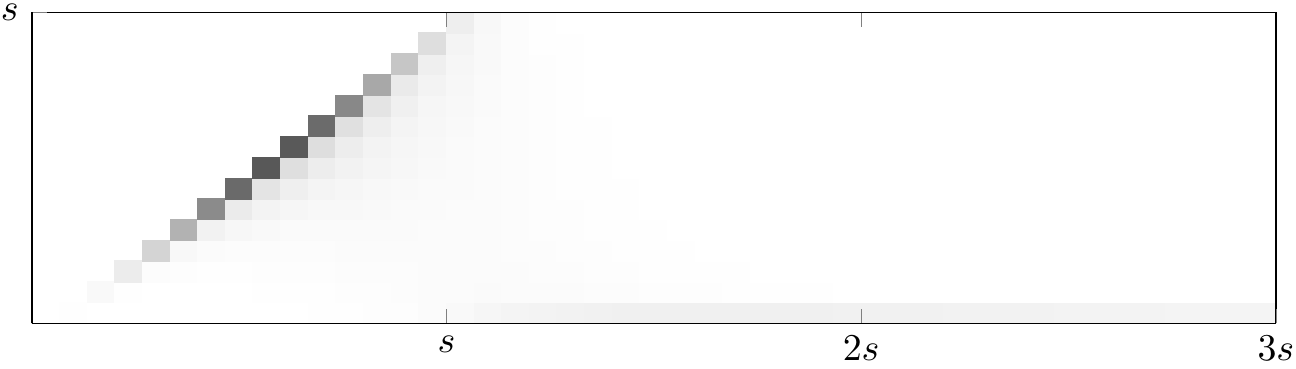}
}\\%
\subfloat[$\rF = 0.95$, $\rho - \rF = 0.027$ and $\Prob{W > 0} = 0.90$.]{%
\includegraphics{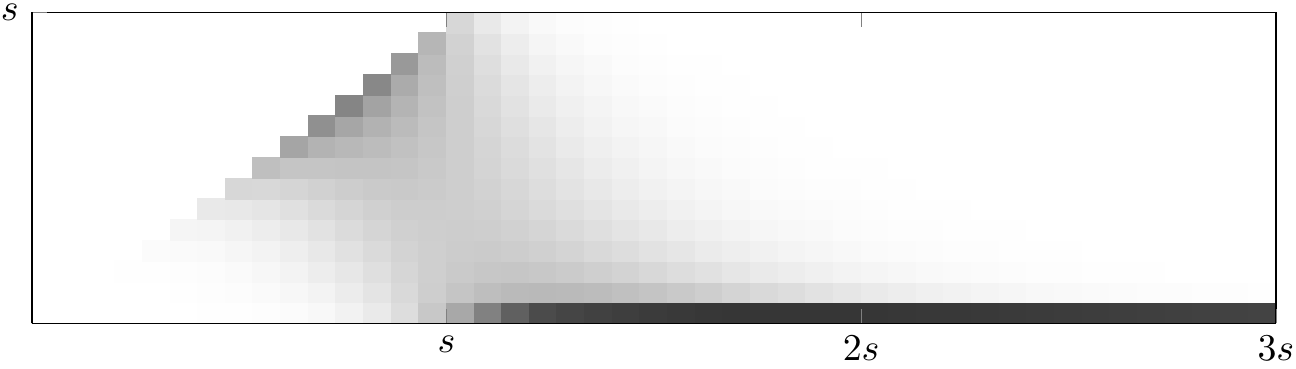}
}%
\caption{For parameter settings that are mild or extreme the slowdown system resembles either the fast or the slow system. The horizontal axis plots the total number of customers in the system and the vertical axis indicates the number of non-delayed customers in service. The contour plot shows where the probability mass is located (darker colour means more mass). Parameter values are $s = 15$ and $\rS = 0.98$.}%
\label{fig:no_bistability}%
\end{figure}%

Arguably the most natural scenario, when $\rF$ is high but not extremely high, say $\rF \in (0.7,0.9)$, gives a less clean picture. Then the slowdown system is a mixture of the fast and slow systems, under the right condition that $\rS$ is decisively larger than $\rF$. A good example is $\rF = 0.8$ and $\rS = 0.98$, as can be seen in Figure~\ref{fig:bistability}(b). By increasing the load $\rS$ busy periods become longer, causing the shift in probability mass towards region 2 and increasing the severity of the slowdown effect in terms of $\rho - \rF$ as is witnessed in Figures~\ref{fig:bistability}(a)-(b).

\begin{figure}%
\centering%
\subfloat[$\rS = 0.9$, $\rho - \rF = 0.047$ and $\Prob{W > 0} = 0.46$.]{
\includegraphics{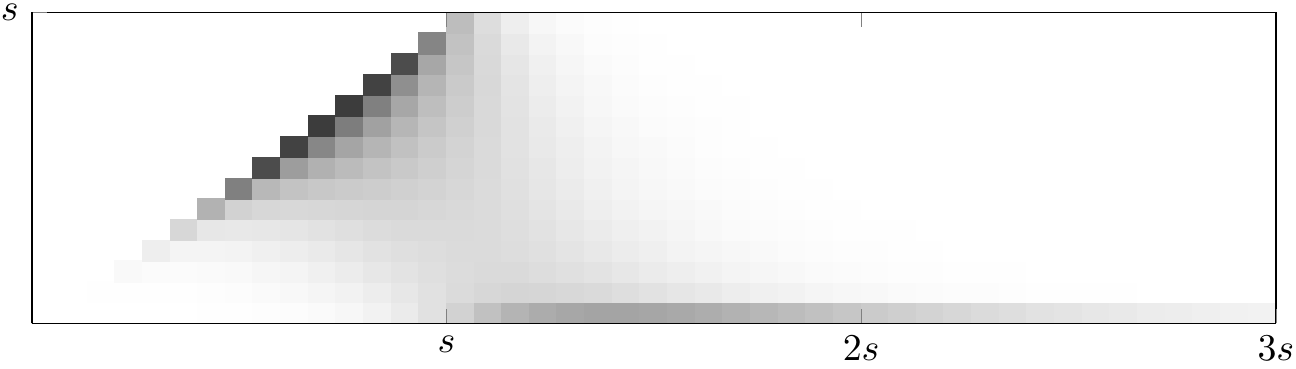}
}\\%
\subfloat[$\rS = 0.98$, $\rho - \rF = 0.144$ and $\Prob{W > 0} = 0.80$.]{
\includegraphics{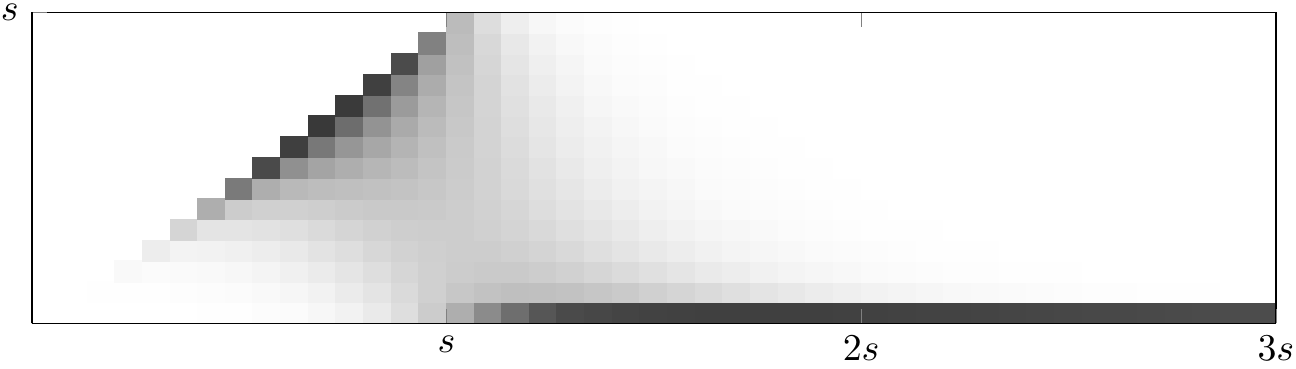}
}%
\caption{Two dominant regions in the stationary distribution become visible when the load of the delayed customers $\rS$ increases. Parameter values are $s = 15$ and $\rF = 0.8$.}%
\label{fig:bistability}%
\end{figure}%

Our slowdown system thus has a subtle bistable behavior, which rises to the surface when both $\rF$ and $\rS - \rF$ are substantial but not extreme. A more extreme bistability effect would occur when $\rS$ could become larger than 1. We therefore next discuss two extensions of our model that allow for $\rS \ge 1$:
\begin{enumerate}[label = \textup{(\roman*)}]%
\item A threshold slowdown system with a finite waiting room;
\item A threshold slowdown system with customer abandonments.
\end{enumerate}%

System (i) can have at most $N$ customers in the system and is therefore inherently stable. When $\rS \ge 1$ and $\rF$ is sufficiently small the system will alternate between periods during which the process settles in the high-occupancy states around $N$, and periods in low-occupancy states below $s$. This gives rise to bistable behavior, and for some parameter ranges even leads to a bimodal distribution as seen in Figure~\ref{fig:extensions}(a). This bimodality can be explained by the fact that for $\rS \ge 1$, the process has two clear points of attraction: the state $N$ and the state $\rF s$ where the rate of arriving and departing customers is equal. Note that our original slowdown system has only one point of attraction, because $\rF < \rS < 1$.

System (ii) assumes that waiting customers abandon the system after an exponential time with mean $1/\delta$. Because the total abandonment rate scales linearly with the number of waiting customers, also this system is inherently stable. For $\rS \ge 1$ it has two points of attraction: one below $s$, and one above $s$ precisely where the total rate of arriving customers equals the rate of departing (abandoning and served) customers. For $\rS \ge 1$ this process alternates between the two points of attraction as is shown in Figure~\ref{fig:extensions}(b). This system is closely related to the operator slowdown system with abandoning customers considered in \cite{Slowdown_Dong2014}. In \cite{Slowdown_Dong2014}, the bistability effect was also observed, where the two points of attraction were identified explicitly. Explicitly characterizing the two points of attraction in the customer slowdown model is more difficult due to the two-dimensional nature of the system.

\begin{figure}%
\centering%
\subfloat[A finite buffer that limits the total number of customers in the system to $N$. Here, $s = 81$, $N = 93$ and $\rF = 0.8$.]{%
\includegraphics{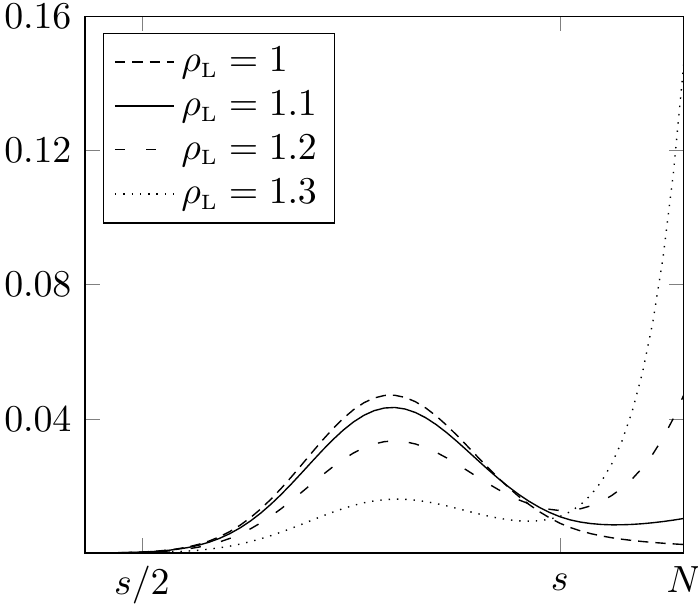}
}%
\subfloat[Allowing customers to abandon the queue with rate $\delta$. Here, $s = 36$, $\rF = 0.7$ and $\rS = 1.2$.]{%
\includegraphics{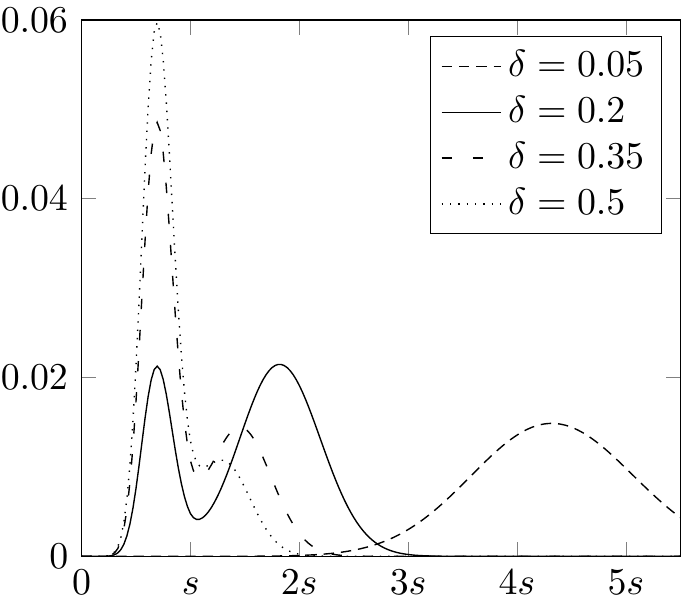}
}%
\caption{Two extensions of the multi-server system with slowdown that exhibit, for a narrow range of parameter values, a bistability effect that is visible in the bimodal marginal distribution of the total number of customers in the system.}%
\label{fig:extensions}%
\end{figure}%



\subsection{Scaling limits}%
\label{subsec:approximations_scaling_limits}%

So far we increase either the number of servers or the arrival rate. We continue by examining a scaling of both parameters at the same time. It is well known in the literature that for $G/M/s$ queues, one should scale the arrival rate or the number of servers such that the load $\rho\seq \sim 1 - \QEDl/\sqrt{s}$ as $s \to \infty$ to achieve QED performance \cite{QED_HalfinWhitt1981}. In terms of our model parameters, the scaling is then as follows:
\begin{align}%
\lseq &= s \mS ( 1 - \QEDl / \sqrt{s} ), \label{eqn:lambda_scaling}
\end{align}%
with constant $\QEDl > 0$ and $s > \QEDl^2$ to guarantee a positive arrival rate $\lseq$. By applying \eqref{eqn:lambda_scaling} to our multi-server system with slowdown one finds that we establish so-called Quality-Driven (QD) performance. QD performance refers to a very high quality of service, e.g.~the probability of delay goes to 0 and many servers are idle. This might be undesirable in view of unnecessary operational costs (overdimensioning). The reason for QD performance is that since $\mF > \mS$ we have $\rF < 1$ in the limit for $s \to \infty$. This ensures that the system stabilizes around a state with relatively low occupancy and with only non-delayed customers in service and no customers waiting in the system. To obtain QED system behavior we set the high service rate according to
\begin{align}
\mFseq &= \mS ( 1 + \QEDmF / \sqrt{s} ), \label{eqn:mu_fast_scaling}
\end{align}%
with constant $\QEDmF > 0$. Note that now $\mFseq/\mS \to 1$ for $s \to \infty$ and thus $\rFseq$ also goes to 1. We refer to the combination of \eqref{eqn:lambda_scaling} and \eqref{eqn:mu_fast_scaling} as a QED-type regime. The reason for this choice of scaling becomes clear when we examine the load of the slowdown system with $s$ servers
\begin{equation}%
\rho\seq = \Bigl(1 - \frac{\QEDl}{\sqrt{s}}\Bigr) \frac{1 + \Prob{W\seq > 0}\frac{\QEDmF}{\sqrt{s}}}{1 + \frac{\QEDmF}{\sqrt{s}}}, \label{eqn:QED_load}
\end{equation}%
which shows that $\rho\seq \uparrow 1$ as $s \to \infty$. Compared to the standard scaling of the load in $G/M/s$ queues, the load in the customer slowdown model approaches 1 slower as it is multiplied by the second term in \eqref{eqn:QED_load}. Figure~\ref{fig:QED_p_wait_s} depicts the probability of delay as a function of $s$, which indeed shows that the probability of delay converges to a value in $(0,1)$.

\begin{figure}%
\centering%
\includegraphics{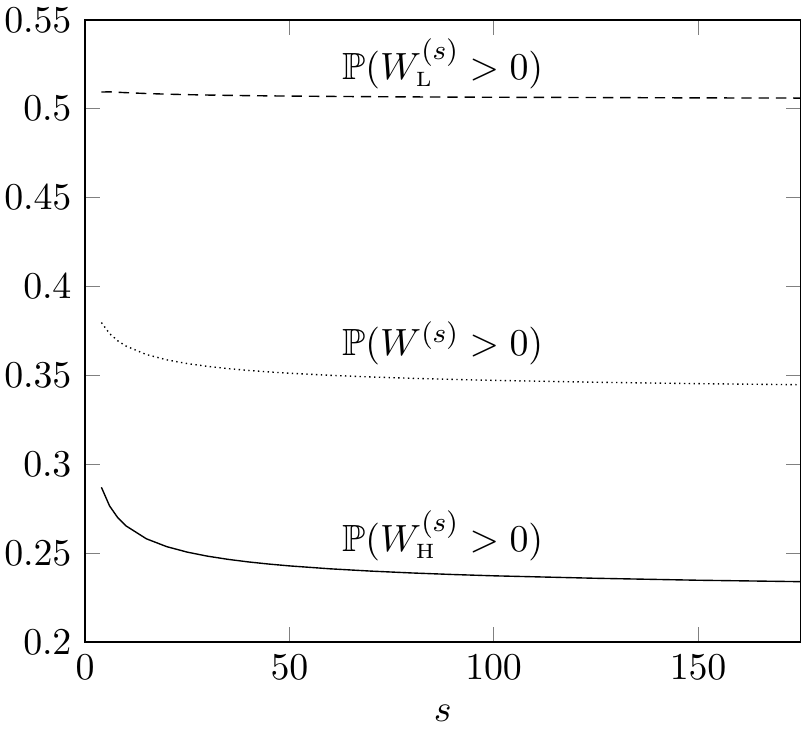}
\caption{Probability to wait for the fast, slow, and customer slowdown system. For all three systems, the scaling \eqref{eqn:lambda_scaling} and \eqref{eqn:mu_fast_scaling} is used with $(\QEDl,\QEDmF) = (0.5,0.5)$.}%
\label{fig:QED_p_wait_s}%
\end{figure}%

Using stochastic coupling techniques, we related the two-dimensional process to two one-dimensional processes that serve as a stochastic lower and upper bound (at the process level in terms of stochastic domination). These two related processes are the fast and the slow system introduced earlier. The bounding processes provide sharp approximations for the two-dimensional process. For both bounding processes, we show that in the QED-type regime, the probability of delay converges to a value strictly in between 0 and 1, and this then also holds for the two-dimensional process. Hence, this provides strong evidence for the existence of a non-trivial stochastic-process limit. Formally establishing the existence and characterizing this stochastic-process limit is a challenging open problem, because the limiting process is likely to be two-dimensional as well, and classical techniques to prove stochastic-process limits \cite{StochasticProcessLimits_Whitt2002} do not seem to carry over easily.

\subsection{Insights}%
\label{subsec:insights}%

Here we summarize the insights obtained in this section.

Customer slowdown of the threshold type leads to severe performance degradation, particularly in heavy traffic. Compared to a system without slowdown, the busy periods are relatively long due to the slowdown of delayed customers that reinforces the persistence of the busy period. We refer to this effect as the snowball effect, which describes the correlated service times when customers are delayed. Further, for a relatively high load $\rF$, we find that the threshold slowdown system is a mixture of the fast and slow systems. This mixture effect is visible in the two-dimensional stationary distribution, where it manifests itself as two dominant regions in terms of probability mass -- a subtle bistable effect. Finally, by using a QED-type scaling for the arrival rate and the fast service rate $\mF$, we have shown that a non-degenerate limit behavior occurs as the number of servers increases.


\section{Model description}%
\label{sec:model}%

Recall that $\X{t} \in \Nat_0$ is the total number of customers in the system at time $t$ and $\Y{t} \in \{0,1,\ldots,s\}$ is the number of non-delayed customers in the system at time $t$. Note that $\X{t} \ge \Y{t}$. Then, $\{(\X{t},\Y{t}), ~ t \ge 0\}$ is an irreducible continuous-time Markov chain with discrete state space $V \cup W$, given by
\begin{equation}%
V = \{(i,j) \mid 0 \le i < s, ~ 0 \le j \le i \}, \quad W = \{(i,j) \mid i \ge s, ~ 0 \le j \le s \}.
\end{equation}%
Recall that we refer to the states with $s$ or less customers in the system, as the \textit{boundary} states. With \textit{inner} states we refer to the states with more than $s$ customers in the system. For an inner state $(i,j)$ with $i > s$, we have the following transition rates:
\begin{itemize}%
\item From $(i,j)$ to $(i+1,j)$ with rate $\la$, $0 \le j \le s$;
\item From $(i,j)$ to $(i-1,j)$ with rate $(s-j) \mS$, $0 \le j \le s$;
\item From $(i,j)$ to $(i-1,j-1)$ with rate $j\mF$, $1 \le j \le s$.
\end{itemize}%
The transition rate diagram of the continuous-time Markov chain is shown in Figure~\ref{fig:trd_abstract}.

Define level $i$ as the set of all states with a total of $i$ customers in the system. Now we have the following alternative description of the transition rates. The matrices $\La_k$ contain the transition rates from level $i$ to level $i + k$ with $i > s$. Let $\I$ be the identity matrix of size $s + 1$. Then the matrices $\La_k$ are given by $\La_1 =  \la \I$,
\begin{align}%
\La_0 = - &\begin{pmatrix}%
\la + s \mS &                     &        &           \\
           & \la + (s - 1)\mS + \mF &        &           \\
           &                     & \ddots &           \\
           &                     &        & \la + s\mF
\end{pmatrix}, \\
\intertext{and}
\La_{-1} = &\begin{pmatrix}%
s \mS &          &        &      &   \\
\mF   & (s - 1)\mS &        &      &   \\
      & 2\mF     & \ddots &      &   \\
      &          & \ddots & \mS  &   \\
      &          &        & s\mF & 0
\end{pmatrix}.%
\end{align}%

By assumption $\rF < \rS$, and we have the following condition for ergodicity of the Markov process.
\begin{lemma}\label{lem:ergodicity_condition}%
The Markov process is ergodic if and only if
\begin{equation}%
\rS < 1. \label{eqn:stability_condition}
\end{equation}%
\end{lemma}
\begin{proof}%
We require that the mean drift in the negative direction is larger than the mean drift in the positive direction; see Neuts' mean drift condition \cite[Theorem~1.7.1]{Classical_Neuts1981}. This gives
\begin{equation}%
\boldsymbol{\pi} \La_1 \oneb < \boldsymbol{\pi} \La_{-1} \oneb,
\end{equation}%
where $\oneb$ is a column vector of ones of size $s + 1$, $\boldsymbol{\pi}$ is the solution of $\boldsymbol{\pi} \sum_{k = -1}^1 \La_k = 0 $ with $\boldsymbol{\pi} \oneb = 1$. We clearly have $\boldsymbol{\pi} = (1,0,\ldots,0)$ and thus the result follows.
\end{proof}%
%


\section{Obtaining the stationary distribution}%
\label{sec:obtaining_stationary_distribution}%

Assume that \eqref{eqn:stability_condition} holds and define the stationary probabilities
\begin{equation}%
p(i,j) \coloneqq \lim_{t \to \infty} \Prob{\X{t} = i, \Y{t} = j}, \quad (i,j) \in V \cup W.
\end{equation}%
The balance equations for the inner states are obtained by equating the rate out of and into an inner state $(i,j)$, yielding, for $i > s, ~ 0 \le j \le s$,
\begin{align}%
(\la + j\mF + (s - j)\mS) p(i,j) &=  \la p(i - 1,j) + (s - j) \mS p(i + 1,j) \notag \\
&\quad + (j + 1) \mF p(i + 1,j + 1), \label{eqn:EE}
\end{align}%
where by convention $p(i,s + 1) = 0$. Equations \eqref{eqn:EE} are referred to as the inner equations. The balance equations for states with $i \le s$ are called the boundary equations.

The stationary probabilities of the inner states are determined using matrix-analytic methods that search for the solution to a non-linear matrix equation. Exploiting structural properties of the Markov process, we derive explicit solutions for these matrices. For similar explicit results using the matrix-analytic methods, see \cite{QBD_Triangular_Houdt2011,QBD_LatticePathCounting_Leeuwaarden2009,QBD_ExplicitR_Leeuwaarden2006}. Next, we solve the boundary equations. Since we want to be able to solve the stationary distribution also for large $s$, solving the $(s + 1)(s + 2)/2$ boundary equations using Gaussian elimination might become computationally cumbersome. We therefore present a more sophisticated algorithm that exploits the structure of the state space and the explicit matrix solution.


\subsection{Inner equations}%
\label{subsec:inner_equations}%

Let $\bld{p}_i = (p(i,0),p(i,1),\ldots,p(i,s))$, and rewrite the inner balance equations as
\begin{equation}%
\bld{p}_{i-1} \La_1 + \bld{p}_i \La_0 + \bld{p}_{i + 1} \La_{-1} = 0, \quad i > s. \label{eqn:EE_vector-matrix_form}
\end{equation}%

The rate matrix $R$ is defined as the minimal non-negative solution of the non-linear matrix equation \cite[Theorem~3.1.1]{Classical_Neuts1981}
\begin{equation}%
\La_1 + R \La_0 + R^2 \La_{-1} = 0. \label{eqn:R_non-linear_matrix_equation}
\end{equation}%
It can be shown that the stationary probabilities satisfy
\begin{equation}%
\bld{p}_i = \bld{p}_s R^{i - s}, \quad i \ge s. \label{eqn:equilibrium_distribution_inner_states}
\end{equation}%
Since the transition matrices are all lower triangular, so is the rate matrix $R$. Denote
\begin{equation}%
R = \begin{pmatrix}%
R_{0,0} &         &        &         \\
R_{1,0} & R_{1,1} &        &         \\
\vdots  &         & \ddots &         \\
R_{s,0} & \cdots  &        & R_{s,s}
\end{pmatrix}%
\end{equation}%
and note that $R^2$ is again a lower triangular matrix with elements $(R^2)_{i,j} = \sum_{k = j}^i R_{i,k} R_{k,j}$ for $i \ge j$.

Equation \eqref{eqn:R_non-linear_matrix_equation} consists of $(s + 1)^2$ separate equations. For the diagonal elements we have
\begin{align}%
\la - (\la + (s - j)\mS + j\mF)R_{j,j} + (s - j)\mS R^2_{j,j} &= 0, \quad 0 \le j < s,  \label{eqn:R_diagonal_elements_j_<_s}\\
\la - (\la + s\mF)R_{s,s} &= 0, \quad j = s, \label{eqn:R_diagonal_elements_j_=_s}
\end{align}%
where $R_{j,j}$ in \eqref{eqn:R_diagonal_elements_j_<_s} is obtained as the minimal non-negative solution. The minimal non-negative solution of \eqref{eqn:R_diagonal_elements_j_<_s} is contained in the interval $(0,1)$, because for $R_{j,j} = 0$ the left-hand side of \eqref{eqn:R_diagonal_elements_j_<_s} is positive, for $R_{j,j} = 1$ the left-hand side of \eqref{eqn:R_diagonal_elements_j_<_s} is negative, and we are dealing with a continuous function. Interestingly, $R_{0,0} = \rS$ and $R_{j,j}$ is monotonically decreasing in $j$. For each element on the subdiagonals we have a linear equation with solution
\begin{equation}%
R_{i,j}  = \frac{\sum_{k = j + 1}^{i - 1} R_{i,k}R_{k,j}(s - j)\mS + \sum_{k = j + 1}^{i} R_{i,k}R_{k,j + 1} (j + 1) \mF}{\la + (s - j)\mS + j\mF - \bigl( R_{i,i}+ R_{j,j} \bigr)(s - j)\mS}, \label{eqn:R_subdiagonal}
\end{equation}%
for $j = i - h, ~ h \le i \le s$ and $h = 1,2,\ldots,s$. In \eqref{eqn:R_subdiagonal} we use the convention that $\sum_{i = i_0}^{i_1} f(i) = 0$ if $i_1 < i_0$. Equations \eqref{eqn:R_diagonal_elements_j_<_s}-\eqref{eqn:R_subdiagonal} fully describe the rate matrix $R$.

Recall that a lower triangular matrix is non-singular if it has all non-zero elements on the diagonal. Thus, the matrix $R$ is non-singular and also $\I - R$ is non-singular. The inverse of $\I - R$ is required to compute the stationary probabilities, as the normalization of the stationary distribution partially follows from $\bld{p}_s (\I + R + R^2 + \cdots) \oneb = \bld{p}_s (\I - R)^{-1} \oneb$. The elements of the inverse are given by
\begin{align}%
((\I - R)^{-1})_{j,j} &= \frac{1}{(\I - R)_{j,j}}, \quad 0 \le j \le s, \\
((\I - R)^{-1})_{i,j} &= \frac{- \sum_{k = j}^{i - 1} (\I - R)_{i,k} ((\I - R)^{-1})_{k,j}}{(\I - R)_{i,i}}, \quad 0 \le j < i \le s.
\end{align}%

Instead of searching for $R$, one can also search for the matrix $G$, defined as the minimal non-negative solution of the non-linear matrix equation
\begin{equation}%
\La_{-1} + \La_0 G + \La_{1} G^2 = 0. \label{eqn:G_non-linear_matrix_equation}
\end{equation}%
The matrices $R$ and $G$ are related as $\La_1 G = R \La_{-1}$ and thus $G = \La_1^{-1} R \La_{-1}$, which exists since $\La_1$ is a diagonal matrix.


\subsection{Boundary equations}%
\label{subsec:boundary_equations}%

The boundary equations can be represented as a set of $(s + 1)(s + 2)/2$ linear equations, which can be solved using Gaussian elimination with an arithmetic complexity of $\BigO(s^6)$ \cite[Chapter~10]{Classical_FraleighBeauregard1995}. By exploiting the structure of the boundary equations one can reduce the arithmetic complexity to $\BigO(s^4)$. In short, we wish to embed the Markov process on level $s$ for which we need the return probabilities when jumping to a higher level (the matrix $G$), combined with the return probabilities when jumping to a level below (yet to be determined). This allows us to compute the non-normalized stationary probabilities of the states in level $s$. Then, we recursively compute the remaining boundary probabilities starting from level $s - 1$, working our way down to level 0, leading to a non-normalized stationary distribution. Finally, the normalized stationary distribution follows by summing over all states and dividing the non-normalized version of the stationary distribution by the resulting sum.

To this end we introduce two first passage probabilities. To aid the derivation of these probabilities we introduce subsets of $V$, indexed by a state $(k,l) \in  V$. We identify the triangular set of states $T_{(k,l)}$ to the South-West of the state $(k,l)$, specifically, $T_{(k,l)} \coloneqq \{ (i,j) \mid k - l \le i \le k - 1, ~ 0 \le j \le i - (k - l) \}$, see Figure~\ref{fig:clarification_triangular_subset_V}.

\begin{figure}%
\centering%
\includegraphics{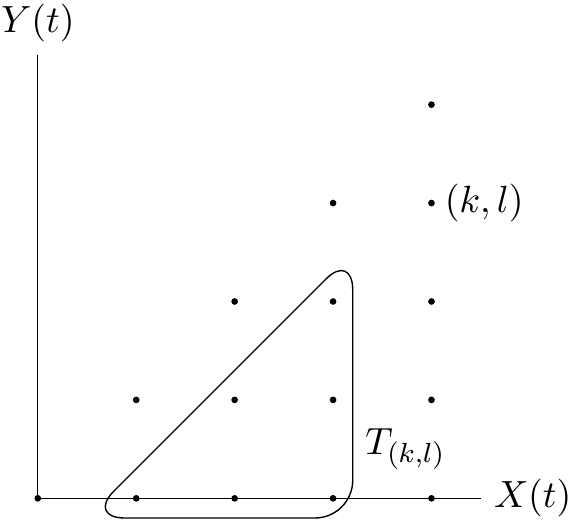}
\caption{A visual clarification of the triangular set of states $T_{(k,l)}$.}%
\label{fig:clarification_triangular_subset_V}%
\end{figure}%

Let $\theta_k(i,j)$ be the first passage probability to phase $j + 1$ in state $(i + 1 - k,j + 1)$, where the Markov process starts in state $(i,j) \in T_{(s,s - k)}$. Note that by phase $j$ we refer to the set of states with $Y(t) = j$. By one-step analysis we obtain
\begin{align}%
\theta_0(i,j) &= \frac{\la \cdot 1 + (i - j)\mS \cdot 0 + j\mF \theta_0(i - 1,j - 1)\theta_0(i,j)}{\la + (i - j)\mS + j\mF}, \quad (i,j) \in T_{(s,s)}, \\
\theta_k(i,j) &= \frac{\la \cdot 0 + (i - j)\mS \theta_{k - 1}(i - 1,j) + j\mF \sum_{l = 0}^k \theta_l(i - 1,j - 1)\theta_{k - l}(i - l,j)}{\la + (i - j)\mS + j\mF}, \notag\\
&\hspace{0.5\linewidth} (i,j) \in T_{(s,s - k)}, ~ k > 0,
\intertext{from which the following expressions can be readily derived,}
\theta_0(i,j) &= \frac{\la}{\la + (i - j)\mS + j\mF(1 - \theta_0(i - 1,j - 1))}, \quad (i,j) \in T_{(s,s)}, \\
\theta_k(i,j) &= \frac{(i - j)\mS\theta_{k - 1}(i - 1,j) + j\mF \sum_{l = 1}^k \theta_{l}(i - 1,j - 1) \theta_{k - l}(i - l,j) }{\la + (i - j)\mS + j\mF(1 - \theta_0(i - 1,j - 1))}, \notag\\
&\hspace{0.5\linewidth} (i,j) \in T_{(s,s - k)}, ~ k > 0.
\end{align}%
Note that $\theta_0(i,i) = 1$, which means that if the Markov process reaches a state on the main diagonal, it eventually always reaches state $(s,s)$.

Let $\psi_{(k,l)}(i,j)$ be the first passage probability to level $k$ in state $(k,l)$, where the Markov process starts from state $(i,j) \in T_{(k,l)}$. We express these first passage probabilities in terms of $\theta_k(i,j)$ as follows
\begin{equation}\label{eqn:first_passage_probability_psi}%
\psi_{(k,l)}(i,j) = \sum_{m = j + 1}^{i + 1} \theta_{i + 1 - m}(i,j) \psi_{(k,l)}(m,j + 1), \quad (i,j) \in T_{(k,l)}.
\end{equation}%
The computation of the first passage probabilities $\psi$ is the most time consuming step in the derivation of the boundary probabilities. Equation \eqref{eqn:first_passage_probability_psi} is evaluated for a total of $s(s + 1)(s + 2)(s + 3)/24$ combinations of $(i,j)$ and $(k,l)$, leading to an arithmetic complexity of $\BigO(s^4)$.

Let $\Psi$ be the matrix of elements $\psi_{(s,k)}(s - 1,j)$ where $j$ is the row number and $k$ the column number. The balance equations of the Markov process embedded at level $s$ are
\begin{equation}\label{eqn:p_s_not_normalized}%
\bld{p}_s \bigl( \La_{-1}\Psi + \La_0 + \La_1 G \bigr) = 0.
\end{equation}%
One solves this homogeneous set of equations using the numerically stable Grassmann version of the Gaussian elimination algorithm \cite{Classical_Grassmann1985} to obtain the stationary probabilities at level $s$. This solution is not normalized.

In order to obtain the remaining boundary probabilities, one embeds the Markov process on the levels $i,i+1,\ldots,s,\ldots$ with $i < s$, for which we derive the following balance equations:
\begin{align}\label{eqn:embedded_level_i_EE}%
&p(i,j)\bigl( \la + (i-j)\mS + j\mF(1-\theta_0(i-1,j-1)) \bigr) \notag \\
&= p(i+1,j)(i+1-j)\mS + p(i+1,j+1)(j+1)\mF \notag \\
&\quad + \sum_{k = 0}^{j - 1} p(i,k) (i-k) \mS \psi_{(i,j)}(i-1,k) + \sum_{k = 1}^{j - 1} p(i,k) k \mF \psi_{(i,j)}(i-1,k-1), \quad j \le i.
\end{align}%
One recursively solves \eqref{eqn:embedded_level_i_EE} by first computing $p(i,0)$, followed by $p(i,1)$, et cetera, until $p(i,i)$ is computed. Then, the remaining boundary probabilities follow by solving \eqref{eqn:embedded_level_i_EE} for $i = s-1,s-2,\ldots,1$. Finally, the probability of having an empty system is found by examining the balance equation in state $(0,0)$, so that
\begin{equation}%
p(0,0) = ( p(1,0) \mS + p(1,1) \mF ) / \la.
\end{equation}%

Recall that the stationary probabilities of level $s$ are not normalized. Thus, the obtained stationary distribution of the boundary and inner states (obtained through \eqref{eqn:equilibrium_distribution_inner_states}) are yet to be normalized. The normalized stationary distribution follows by dividing each non-normalized stationary probability by the sum over all states $\sum_{(i,j) \in V} p(i,j) + \bld{p}_s (\I - R)^{-1} \oneb$.

Using the stationary distribution, one can now obtain performance measures such as the delay probability
\begin{equation}%
\Prob{W > 0} = \sum_{i = 0}^\infty \bld{p}_{s + i} \oneb = \bld{p}_s (\I - R)^{-1} \oneb,
\end{equation}%
or the mean queue length
\begin{equation}%
\E{Q} = \sum_{i = 0}^\infty i \bld{p}_{s + i} \oneb = \bld{p}_s R (\I - R)^{-2} \oneb.
\end{equation}%
%


\subsection{Extensions}%
\label{subsec:extensions_stationary_distribution}%

We next describe how to obtain the stationary distribution of the slowdown model with (i) a finite buffer; or (ii) customer abandonments.


\subsubsection{Finite buffer} %
\label{subsubsec:finite_buffer}%

The transition rate diagram of the slowdown system with a finite buffer is identical to the one shown in Figure~\ref{fig:trd_abstract} but truncated at level $N$. Recall that we defined the matrices $\La_k$ to contain the transition rates from level $i$ to level $i + k$; since we now introduced the finite buffer, we restrict $i$ as $s < i \le N$. At level $N$, the matrix containing the transitions rates to level $N - 1$ remains unchanged and is still $\La_{-1}$. The only difference is that there are no jumps in the positive direction and thus $\La_1$ does not exist and therefore the matrix $\La_0$ changes at level $N$, now indexed by an additional subscript $N$ and given by $\La_{0,N} = \La_0 + \La_1$.

The equilibrium equations in vector-matrix form are given by
\begin{align}%
\bld{p}_{i - 1} \La_1 + \bld{p}_i \La_0 + \bld{p}_{i + 1} \La_{-1} &= 0, \quad s < i < N, \label{eqn:finite_buffer_EE_strip} \\
\bld{p}_{N - 1} \La_1 + \bld{p}_N \La_{0,N} &= 0. \label{eqn:finite_buffer_EE_boundary}
\end{align}%
We now have the following relation, see \cite[Section~2.2]{FiniteStateQBD_Elhafsi2007},
\begin{equation}%
\bld{p}_i = \bld{p}_{i - 1} R_i, \quad s < i \le N, \label{eqn:p_i}
\end{equation}%
where $R_i$ is a level-dependent rate matrix with identical interpretation as the standard rate matrix of the matrix-geometric approach. One can now solve for the rate matrix $R_N$ by manipulating \eqref{eqn:finite_buffer_EE_boundary} as follows
\begin{equation}%
\bld{p}_N  = - \bld{p}_{N - 1} \La_1 \bigl( \La_{0,N} \bigr)^{-1} = \bld{p}_{N - 1} R_N.
\end{equation}%
Note that $\La_{0,N}$ is a diagonal matrix with negative elements and is therefore indeed non-singular. The remaining rate matrices are found from \eqref{eqn:finite_buffer_EE_strip} as
\begin{equation}%
\bld{p}_i = - \bld{p}_{i - 1} \La_1 \Bigl( \La_0 + R_{i + 1} \La_{-1} \Bigr)^{-1} = \bld{p}_{i - 1} R_i, \quad s < i < N.
\end{equation}%
The matrix $\La_0 + R_{i + 1} \La_{-1}$ is lower triangular with negative elements on the diagonal and is therefore non-singular; for the proof of this statement, see \cite[p.~519]{Level-DependentQBD_Bright1995}.

This leaves us to compute $\bld{p}_s$ and the equilibrium probabilities of the boundary states. We do so with the approach we have derived earlier for the slowdown system with an infinite buffer. The missing ingredients are the first passage probabilities from level $s+1$ to level $s$, which are found through the relation
\begin{equation}%
G_{i} = \La_1^{-1} R_{i} \La_{-1}, \quad i > s. \label{eqn:level_dependent_relation_G_and_R}
\end{equation}%
Note that the auxiliary matrices $G_i$ are level-dependent and have the same interpretation as the standard auxiliary matrix in the matrix-analytic approach. Thus, we substitute the level-dependent matrix $G_{s + 1}$ for $G$ in \eqref{eqn:p_s_not_normalized} and are able to compute the complete stationary distribution.


\subsubsection{Customer abandonments}%
\label{subsubsec:customer_abandonments}%

The base model is appended by adding transitions with rate $l\delta$ from state $(s + l,j)$ to state $(s + l - 1,j)$ for $l > 0$. These transitions model a waiting customer abandoning the queue. By appending the base model with these transitions a level-dependent QBD (LDQBD) process is created. We use solution techniques for LDQBD processes as presented in \cite{Level-DependentQBD_Bright1995,Level-DependentQBD_Kharoufeh2011} to compute the stationary distribution of the semi-infinite strip of states and once again use the earlier derived technique to compute the equilibrium distribution of the boundary states. We briefly sketch the solution approach here.

The aggregated abandonment rate depends on the number of customers waiting in the queue. This leads to level-dependent transition rate matrices which we label with an additional subscript $l$, such that $\La_{k,l}$ describes the transition rates from level $s + l$ to level $s + l - k$ with $l > 0$. The transition rate matrices are given by $\La_{1,l} = \La_1$, $\La_{0,l} = \La_0 - l\delta \I$ and $\La_{-1,l} = \La_{-1} + l\delta \I$.

The solution approach is based on the same premise as for the finite QBD process case, namely
\begin{equation}%
\bld{p}_i = \bld{p}_{i - 1} R_i, \quad i > s,
\end{equation}%
where $R_i$ is a level-dependent rate matrix with identical interpretation as the standard rate matrix of the matrix-geometric approach.

The following is explained in greater detail in \cite{Level-DependentQBD_Bright1995}. Since generally only numerical solutions can be found for the $R_i$ matrices of LDQBD processes, one resorts to computing the sequence $\{ R_i \}_{s < i \le N^*}$, where $N^*$ is chosen ``large enough''. By \cite[Lemma~1]{Level-DependentQBD_Bright1995} we have the explicit expression
\begin{equation}%
R_i = \sum_{j = 0}^\infty U_i^j \prod_{k = 0}^{j - 1} D_{i + 2^{j - k}}^{j - 1 - k}, \quad i > s, \label{eqn:abandonments_R_i_explicit}
\end{equation}%
where $U_i^j$ and $D_i^j$ are matrices defined recursively and are a function of the level-dependent transition matrices. Truncating the infinite sum in \eqref{eqn:abandonments_R_i_explicit}, one computes $R_{N^*}$. The remaining rate matrices then follow from the standard relation
\begin{equation}%
R_i = -\La_{1,i - s - 1} \bigl( \La_{0,i - s} + R_{i + 1} \La_{-1,i - s + 1} \bigr)^{-1}, \quad i > s. \label{eqn:abandonments_R_i_recursive}
\end{equation}%
Note that the inverse exists.

As in the finite buffer case, this leaves us to compute $\bld{p}_s$ and the equilibrium probabilities of the boundary states. Once again, the first passage probabilities $G_{s+1}$ are needed and follow from \eqref{eqn:level_dependent_relation_G_and_R}. Then, we substitute the level-dependent matrix $G_{s + 1}$ for $G$ in \eqref{eqn:p_s_not_normalized} and are able to compute the complete stationary distribution.

\section{A QED-type regime}%
\label{sec:scaling limits}%

We next analyze the behavior of the multi-server queueing system incorporating slowdown for large $s$ and $\rS \to 1$ by considering a sequence of queues, indexed by $s$. We write $(\X{t},\Y{t})) = (\Xseq{t},\Yseq{t})$, $\la = \lseq$, $\mF = \mFseq$, and $\rS = \rSseq$. Without loss of generality we keep $\mS$ constant and assume throughout that $\mFseq > \mS$.

Let $\pdseq$ denote the probability that a customer has to wait in a slowdown system with $s$ servers. We will identify a regime in which $\pdseq \to \pd \in (0,1)$ so that the limiting system displays non-degenerate behavior, as in the classical QED regime. In order to do so, we introduce a random variable $\XFseq{t}$ that represents the total number of customers at time $t$ in an $M/M/s$ queue where all customers are served with the high service rate $\mFseq$. As we have done before, we refer to this queueing system as the \textit{fast} system. Let the random variable $\XSseq{t}$ represent the total number of customers at time $t$ in an $M/M/s$ queue where all customers are served with the low service rate $\mS$. We refer to this queueing system as the \textit{slow} system.

For two real-valued random variables $A$ and $B$, we say that $A$ stochastically dominates $B$ if
\begin{equation}%
\Prob{A \le x} \le \Prob{B \le x},
\end{equation}%
and we denote this as $A \succeq B$. The following result is proved in Appendix~\ref{app:proof_stochastic_dominance}.
\begin{lemma}\label{lem:stochastic_dominance}%
$\XSseq{t} \succeq \Xseq{t} \succeq \XFseq{t}.$
\end{lemma}%

We next introduce the scaling
\begin{align}%
\lseq &= s \mS - \QEDl \mS \sqrt{s}, \label{eqn:lambda_scaling_maths_section} \\
\mFseq &= \mS ( 1 + \QEDmF / \sqrt{s} ), \label{eqn:mu_fast_scaling_maths_section}
\end{align}%
with constants $\mS, \QEDl,\QEDmF > 0$ and $s \ge \QEDl^2$. Note that $\mFseq/\mS \to 1$ for $s \to \infty$. We refer to the scaling \eqref{eqn:lambda_scaling_maths_section} and \eqref{eqn:mu_fast_scaling_maths_section} as a QED-type scaling regime. We introduce the scaled random variables
\begin{equation}%
\Dseq{t} \coloneqq \frac{\Xseq{t} - s}{\sqrt{s}}, ~ \DFseq{t} \coloneqq \frac{\XFseq{t} - s}{\sqrt{s}}, ~ \DSseq{t} \coloneqq \frac{\XSseq{t} - s}{\sqrt{s}} . \label{eqn:Xseq_customer_slowdown_scaling}
\end{equation}%
Note that Lemma~\ref{lem:stochastic_dominance} also holds for these scaled random variables, i.e.~$\DSseq{t} \succeq \Dseq{t} \succeq \DFseq{t}$. The following lemma is proved in Appendix~\ref{app:proof_mean_variance_diffusion_processes_bounds_eta_1}.
\begin{lemma}\label{lem:mean_variance_diffusion processes_bounds_eta_1}%
If $\DFseq{0} = d\F\seq$ and $\DF{0} = d\F$ \textup{a.s.} with $d\F\seq \to d\F$, and $\DSseq{0} = d\SL\seq$ and $\DS{0} = d\SL$ \textup{a.s.} with $d\SL\seq \to d\SL$, then for $s \to \infty$, and for every $t \ge 0$, the scaled random variables $\DFseq{t} \dconv \DF{t}$ and $\DSseq{t} \dconv \DS{t}$, where the infinitesimal means of the diffusion processes are given by
\begin{equation}%
\meanDF = \begin{cases}%
\mS(-\QEDl - \QEDmF - x), & x \le 0, \\
\mS(-\QEDl - \QEDmF), & x > 0,
\end{cases} \quad%
\meanDS = \begin{cases}%
\mS ( -\QEDl - x ), & x \le 0, \\
-\QEDl\mS, & x > 0,
\end{cases}%
\end{equation}%
and constant infinitesimal variances $\varDF = \varDS = 2 \mS$.
\end{lemma}%
\begin{remark}%
Both processes $\DF{\cdot}$ and $\DS{\cdot}$ behave as an Ornstein-Uhlenbeck process below level zero and a reflected Brownian motion above level zero.
\end{remark}%
\begin{corollary}\label{cor:pdf_diffusion processes_bounds}%
The probability density functions of $\DF{\infty}$ and $\DS{\infty}$ are given by
\begin{equation}%
\pdfF{x} = \begin{cases}%
\pdfCF \frac{\phi(x + \QEDl + \QEDmF)}{\Phi(\QEDl + \QEDmF)}, & x \le 0, \\
(1 - \pdfCF) (\QEDl + \QEDmF)\euler^{-(\QEDl + \QEDmF)x}, & x > 0,
\end{cases}, \quad%
\pdfS{x} = \begin{cases}%
\pdfCS \frac{\phi(x + \QEDl)}{\Phi(\QEDl)}, & x \le 0, \\
(1 - \pdfCS) \QEDl\euler^{-\QEDl x}, & x > 0,
\end{cases}%
\end{equation}%
with
\begin{equation}%
\pdfCF = \frac{\QEDl + \QEDmF}{\QEDl + \QEDmF + \frac{\phi(\QEDl + \QEDmF)}{\Phi(\QEDl + \QEDmF)}}, \quad
\pdfCS = \frac{\QEDl}{\QEDl + \frac{\phi(\QEDl)}{\Phi(\QEDl)}},
\end{equation}%
and $\phi(x)$ and $\Phi(x)$ the probability density function and cumulative density function of the standard normal distribution.
\end{corollary}%
\begin{proof}%
Since we are dealing with piecewise-linear diffusion processes, one computes the probability density functions using \cite[Sections~1 and 4]{DiffusionProcesses_BrowneWhitt1995}.
\end{proof}%
\begin{remark}\label{rem:interchange_of_limits_diffusion_processes}%
The stationary distribution of the diffusion process related to the fast system is equal to the distribution of the limiting random variable of the sequence $(\XFseq{\infty} - s)/\sqrt{s}$ as shown in \cite[Corollary~2]{QED_HalfinWhitt1981}, which establishes that an interchange of limits is allowed. Thus, one can use $\XFseq{\infty} \approx s + \DF{\infty}\sqrt{s}$. The same applies for the slow system.
\end{remark}%
\begin{corollary}\label{cor:bounds_probability_waiting}%
The limiting probability of delay in the slowdown system $\pdseq \to \pd \in (0,1)$ for $s \to \infty$ and can be bounded as follows
\begin{equation}%
\Bigl( 1 + \QEDl \frac{\Phi(\QEDl)}{\phi(\QEDl)} \Bigr)^{-1} = \pdF \ge \pd \ge \pdS = \Bigl( 1 + (\QEDl + \QEDmF) \frac{\Phi(\QEDl+\QEDmF)}{\phi(\QEDl+\QEDmF)} \Bigr)^{-1}.
\end{equation}%
\end{corollary}%
\begin{proof}%
The limiting probability of delay in the fast system is computed from the distribution of $\DF{\infty}$ as
\begin{equation}%
\pdF = \int_0^\infty \pdfF{x} \,\dinf x = 1 - \pdfCF \in (0,1)
\end{equation}%
and identically for the slow system to get $\pdS = 1 - \pdfCS \in (0,1)$. Using Lemma~\ref{lem:stochastic_dominance} we find that these are lower and upper bounds on the limiting probability to wait $\pd$, respectively.
\end{proof}%


\section{Conclusion}%
\label{sec:conclusion}%

We have studied a threshold slowdown system in a Markovian setting. The threshold slowdown system incorporates a slowdown effect in which customers that are delayed require a longer service time. A delayed customer requiring a longer service time will increase the overall workload in the system, therefore causing longer delays for other customers, who in turn due to slowdown also require a longer service time. We refer to this phenomenon as the snowball effect. The snowball effect has been shown to be the leading cause of a severe performance degradation and the neglect of slowdown might lead to underprovisioning. A subtle bistable behavior is witnessed for slowdown systems with relevant parameter settings: the slowdown system either has only non-delayed customers and no customers waiting, or only delayed customers with many customers waiting, switching between configurations over time. We have introduced a QED-type regime for the slowdown system with many-servers and established non-degenerate limiting behavior. The snowball effect has been shown to persist in this QED-type regime.


\appendix%


\section{Proof of Lemma~\ref{lem:stochastic_dominance}}%
\label{app:proof_stochastic_dominance}%

The proof is based on a coupling argument and follows the same reasoning as \cite[Appendix~B]{Slowdown_Chan2013}. We distinguish between two cases: (i) $\XSseq{t} \succeq \Xseq{t}$; and (ii) $\Xseq{t} \succeq \XFseq{t}$. Recall that for two real-valued random variables $A$ and $B$, we say that $A$ first-order stochastically dominates $B$ if
\begin{equation}%
\Prob{A \le x} \le \Prob{B \le x}.
\end{equation}%

(i) Assume that both queues see a common arrival process. Let the service time for the $i$-th arriving customer in the slow system be $\BS(i)$; the corresponding service time in the slowdown model is then either $B(i) = \BS(i)$ or $B(i) = \mS/\mFseq \BS(i)$ depending on whether the slowdown model has high ($s$ or more customers in the system) or low congestion (less than $s$ customers in the system) upon arrival of the $i$-th customer. Finally, we assume that both systems start empty. Let $W(i)$ and $\WS(i)$ denote the waiting time of the $i$-th arriving customer before beginning service in the slowdown and slow system, respectively. We have the following result.
\begin{lemma}\label{lem:stochastic_dominance_coupling_case_1}%
$\WS(i) \ge W(i)$ \textup{a.s.} for all $i$. Moreover, $\XSseq{t} \ge \Xseq{t}$ \textup{a.s.} for all $t$.
\end{lemma}%
\begin{proof}%
Let us prove the first statement using induction and fix an arbitrary event in the sample space $\omega \in \Omega$ leading to a sample path of the process. We append the argument $\omega$ to the variables to indicate that we are studying a fixed sample path. Since we start with an empty system, observe that for the first customer we have $\WS(1,\omega) = W(1,\omega) = 0$. Assume that the statement is true for the $j$-th arriving customer and consider the $(j + 1)$-th arriving customer. For the sake of contradiction assume $\WS(j + 1,\omega) < W(j + 1,\omega)$. When customer $j + 1$ starts service in the slow system:
\begin{itemize}%
\item There are at most $s - 1$ customers among the first $j$ arriving customers present in the slow system.
\item At least $s$ customers from among the first $j$ arriving customers are still present in the slowdown system since customer $j + 1$ has not yet started service in the slowdown system.
\end{itemize}%
From these two observations we conclude that there is a customer among the first $j$ arriving customers that finished service strictly earlier in the slow system than in the slowdown system. However, due to the coupling we have $B(i,\omega) \le \BS(i,\omega), ~ i = 1,\ldots,j$ and thus we have a contradiction. We have consequently established that $\WS(i,\omega) \ge W(i,\omega)$ for all $i$. Recall that we fixed an arbitrary event and thus it holds for all $\omega \in \Omega$. The latter statement of the proposition follows immediately.
\end{proof}%
Lemma~\ref{lem:stochastic_dominance_coupling_case_1} indeed shows that $\XSseq{t} \succeq \Xseq{t}$.

(ii) This case follows using the exact same reasoning as for case (i) and we thus omit the proof.


\section{Proof of Lemma~\ref{lem:mean_variance_diffusion processes_bounds_eta_1}}%
\label{app:proof_mean_variance_diffusion_processes_bounds_eta_1}%

The following proof is based on the proof in \cite[Proposition~3.2]{DiffusionProcesses_Janssen2013}. We first describe convergence in distribution for a general sequence of birth--death processes and apply these results to our processes of interest.

Define $A\seq(\cdot)$ as a continuous-time birth--death process with state space $E\seq = \{ a\seq(i) \mid 0 \le i < \infty \}$, where $a\seq(i)$ is increasing in $i$ and $\lim_{i \to \infty} a\seq(i) = a\seq(\infty)$. Then let
\begin{equation}%
e\seq(x) = \underset{i \in \Nat_0}{\argsup}\{a\seq(i) : a\seq(i) \le x \}, \quad x \in [ a\seq(0), a\seq(\infty)).
\end{equation}%
Denote by $\la\seq(a\seq(i))$ and $\mu\seq(a\seq(i))$ the birth--death parameters associated with state $a\seq(i)$. The infinitesimal mean and variance of the process $A\seq(\cdot)$ are given by
\begin{align}%
m\seq(x) &= \la\seq(e\seq(x)) \Bigl( a^{(s)}(e^{(s)}(x) + 1) - a^{(s)}(e^{(s)}(x)) \Bigr) \notag \\
&\quad - \mu^{(s)}(e^{(s)}(x)) \Bigl( a^{(s)}(e^{(s)}(x)) - a\seq(e\seq(x)-1) \Bigr), \label{eqn:infinitesimal_mean_example_process} \\
\bigl( \sigma^2 \bigr)\seq(x) &= \la\seq(e\seq(x)) \Bigl( a\seq(e\seq(x) + 1) - a\seq(e\seq(x)) \Bigr)^2 \notag \\
&\quad + \mu\seq(e\seq(x)) \Bigl( a\seq(e\seq(x)) - a\seq(e\seq(x)-1) \Bigr)^2, \label{eqn:infinitesimal_variance_example_process}
\end{align}%
whenever $x \in [ a\seq(0), a\seq(\infty))$.

Stone's theorem \cite[Theorem~5.1]{DiffusionProcesses_Iglehart1974} then establishes convergence in distribution of the sequence of Markov processes to a limiting diffusion process.
\begin{theorem}[Stone]\label{thm:Stone_convergence_distribution_diffusion_process}%
Let $A\seq(0) = a\seq$ and $A(0) = a$ \textup{a.s.}, with $a\seq \to a$. Then the following two conditions are sufficient for $A\seq \dconv A$ as elements of $D[0,\infty)$\textup{:}
\begin{enumerate}[label = \textup{(\roman*)}]%
\item $E\seq$ becomes dense in $\Real$ as $s \to \infty$;
\item For every compact subinterval $U$ of $\Real$
    \begin{equation}%
    \lim_{s \to \infty} m\seq(x) = m(x), \quad \lim_{s \to \infty} \bigl( \sigma^2 \bigr)\seq(x) = \sigma^2(x),
    \end{equation}%
    uniformly for $x \in U$.
\end{enumerate}%
\end{theorem}%

We first focus on the fast system with scaled process $\DFseq{\cdot}$ and state space $\SSFseq = \{ (i - s)/\sqrt{s} \mid i \in \Nat_0 \}$. Naturally, $\lim_{s \to \infty} \SSFseq$ is dense in $\Real$, by which we mean that $\forall x \in \Real, \forall \epsilon > 0, \exists s > 0$ such that $\inf_{y \in \SSFseq} \vert x - y \vert < \epsilon$, thus satisfying condition (i) of Theorem~\ref{thm:Stone_convergence_distribution_diffusion_process}. Note that for this state space $e\seq(x) = \lfloor s + x \sqrt{s} \rfloor, ~ x \in [-\sqrt{s},\infty)$.

Now, use that for the fast (and slow) system the expression $a\seq(e\seq(x) + 1) - a\seq(e\seq(x))$ in \eqref{eqn:infinitesimal_mean_example_process} and \eqref{eqn:infinitesimal_variance_example_process} is equal to $1/\sqrt{s}$ to obtain the infinitesimal mean and variance of the process $\DFseq{\cdot}$ as
\begin{align}%
\meanDFseq &= \begin{cases}%
\frac{1}{\sqrt{s}} \bigl(\lseq - \lfloor s + x\sqrt{s} \rfloor \mFseq \bigr), & x \le 0, \\
\frac{1}{\sqrt{s}} \bigl(\lseq - s \mFseq \bigr), & x > 0,
\end{cases} \\
\varDFseq &= \begin{cases}%
\frac{1}{s} \bigl(\lseq + \lfloor s + x\sqrt{s} \rfloor \mFseq \bigr), & x \le 0, \\
\frac{1}{s} \bigl(\lseq + s \mFseq \bigr), & x > 0.
\end{cases}%
\end{align}%

We use the scaling \eqref{eqn:lambda_scaling_maths_section} and \eqref{eqn:mu_fast_scaling_maths_section} to obtain
\begin{align}%
\meanDFseq &= \begin{cases}%
\mS \bigl( -\QEDl - \frac{\lfloor s + x\sqrt{s} \rfloor}{s} \QEDmF + \frac{s - \lfloor s + x\sqrt{s} \rfloor}{\sqrt{s}} \bigr), & x \le 0, \\
\mS \bigl( -\QEDl - \QEDmF \bigr), & x > 0,
\end{cases} \\
\varDFseq &= \begin{cases}%
\mS \bigl( 1 + \frac{\lfloor s + x\sqrt{s} \rfloor}{s} - \frac{\QEDl}{\sqrt{s}} + \frac{\lfloor s + x\sqrt{s} \rfloor}{s} \frac{\QEDmF}{\sqrt{s}} \bigr), & x \le 0, \\
\mS \bigl( 2 - \frac{\QEDl - \QEDmF}{\sqrt{s}} \bigr), & x > 0.
\end{cases}%
\end{align}%
By requiring that $\lim_{s \to \infty} \meanDFseq$ and $\lim_{s \to \infty} \varDFseq$ are finite, we indeed find that $\QEDl$ and $\QEDmF$ can be any value larger than 0. For every compact subinterval $U$ of $\Real$, $\lim_{s \to \infty} \meanDFseq = \meanDF$ and $\lim_{s \to \infty} \varDFseq = \varDF$ uniformly for $x \in U$. Thus, condition (ii) of Theorem~\ref{thm:Stone_convergence_distribution_diffusion_process} is satisfied and $\DFseq{t} \dconv \DF{t}$.

Next we turn to the slow system with the associated scaled process $\DSseq{\cdot}$. Its state space is equal to $\SSFseq$ and thus in the limit for $s \to \infty$ also dense in $\Real$. Using the scaling as proposed in \eqref{eqn:lambda_scaling_maths_section} and \eqref{eqn:mu_fast_scaling_maths_section}, the infinitesimal mean and variance of the process $\DSseq{\cdot}$ are
\begin{align}%
\meanDSseq &= \begin{cases}%
\mS \bigl( -\QEDl + \frac{s - \lfloor s + x\sqrt{s} \rfloor}{\sqrt{s}} \bigr), & x \le 0, \\
-\QEDl \mS, & x > 0, \\
\end{cases} \\
\varDSseq &= \begin{cases}%
\mS \bigl( 1 + \frac{\lfloor s + x\sqrt{s} \rfloor}{s} - \frac{\QEDl}{\sqrt{s}} \bigr), & x \le 0, \\
\mS \bigl( 2 - \frac{\QEDl}{\sqrt{s}} \bigr), & x > 0.
\end{cases}%
\end{align}%
For every compact subinterval $U$ of $\Real$, $\lim_{s \to \infty} \meanDSseq = \meanDS$ and $\lim_{s \to \infty} \varDSseq = \varDS$ uniformly for $x \in U$. Again, conditions (i) and (ii) of Theorem~\ref{thm:Stone_convergence_distribution_diffusion_process} are satisfied and $\DSseq{t} \dconv \DS{t}$.


\subsubsection*{Acknowledgement}%

This work was supported by a free competition grant from NWO and an ERC starting grant.


\bibliographystyle{plain}%
\bibliography{CustomerSlowdownBibRevision}%

\end{document}